\documentclass{article}

\usepackage[margin = 2.5cm]{geometry}
\usepackage{setspace}

\usepackage[british]{babel}
\usepackage[applemac]{inputenc}
\usepackage{amsmath}
\usepackage{mathtools}
\usepackage{amssymb}
\usepackage{amsthm}
\usepackage{mathrsfs}
\usepackage{xcolor}
\usepackage{enumitem}
\usepackage{geometry}
\usepackage{hyperref}
\usepackage[nameinlink]{cleveref}
\usepackage{framed}
\usepackage{bm}
\usepackage{bbm}
\usepackage[square,sort,comma,numbers]{natbib}
\usepackage{dsfont}
\newcommand{\allonesvector}[0]{\mathds{1}}

\newtheorem{theorem}{Theorem}
\newtheorem{claim}[theorem]{Claim}
\newtheorem{lemma}[theorem]{Lemma}
\newtheorem{corollary}[theorem]{Corollary}

\theoremstyle{remark}
\newtheorem*{rem}{Remark}

\theoremstyle{definition}
\newtheorem{definition}[theorem]{Definition}

\usepackage{xpatch}
\xpatchcmd{\proof}{\itshape}{\normalfont\proofnamefont}{}{}

\newcommand{\proofnamefont}{}

\renewcommand{\proofnamefont}{\bfseries}

\newcommand{\C}{\ensuremath{\mathbb{C}}}

\newcommand{\E}{\ensuremath{\mathbb{E}}}
\newcommand{\Fe}{\ensuremath{\mathbb{F}}}

\newcommand{\N}{\ensuremath{\mathbb{N}}}

\newcommand{\R}{\ensuremath{\mathbb{R}}}

\DeclareMathOperator{\rk}{rk}
\DeclareMathOperator{\tr}{tr}

\DeclareMathOperator{\ex}{ex}

\newcommand{\bangle}[1]{\left\langle #1 \right\rangle}
\newcommand{\inprod}[2]{\bangle{#1, #2}}

\DeclareMathOperator{\msr}{msr}
\newcommand{\norm}[1]{\left|\left| #1 \right|\right|}

\title{Orthonormal representations of $H$-free graphs}
\author{Igor Balla\thanks{Department of Mathematics, ETH Zurich, Switzerland. Email:
    \href{mailto:igor.balla@math.ethz.ch} {\nolinkurl{igor.balla@math.ethz.ch}}.}
    \and
    Shoham Letzter\thanks{ETH Institute for Theoretical Studies, ETH Zurich, Switzerland. Email: \href{mailto:shoham.letzter@eth-its.ethz.ch} {\nolinkurl{shoham.letzter@eth-its.ethz.ch}}.
    Research supported by Dr. Max R\"ossler, the Walter Haefner Foundation and the ETH Zurich Foundation.}
    \and
    Benny Sudakov\thanks{Department of Mathematics, ETH Zurich, Switzerland. Email:
    \href{mailto:benjamin.sudakov@math.ethz.ch} {\nolinkurl{benjamin.sudakov@math.ethz.ch}}.
    Research supported in part by SNSF grant 200021-175573.}
}
\date{}

\begin{document}
\maketitle

\begin{abstract}
    Let $x_1, \ldots, x_n \in \R^d$ be unit vectors such that among any three there is an orthogonal pair. How large can $n$ be as a function of $d$, and how large can the length of $x_1 + \ldots + x_n$ be? The answers to these two celebrated questions, asked by Erd\H{o}s and Lov\'{a}sz, are closely related to orthonormal representations of triangle-free graphs, in particular to their Lov\'{a}sz $\vartheta$-function and minimum semidefinite rank. In this paper, we study these parameters for general $H$-free graphs. In particular, we show that for certain bipartite graphs $H$, there is a connection between the Tur\'{a}n number of $H$ and the maximum of $\vartheta \left( \overline{G} \right)$ over all $H$-free graphs $G$.

\end{abstract}

\section{Introduction}
    
    Given a graph $G$, a map $f : V(G) \rightarrow \R^d$ is called an \emph{orthonormal representation} of $G$ (in $\R^d$) if $||f(u)|| = 1$ for all $u \in V(G)$ and $\inprod{f(u)}{f(v)} = 0$ for all distinct $u, v \in V(G)$ such that $uv \notin E(G)$. Note that every graph $G$ on $n$ vertices has an orthonormal representation, since we may assign each vector to a corresponding orthonormal basis vector in $\R^{|G|}$. Given an orthonormal representation $f$ of a graph $G$ with vertex set $[n]$, we define $M_f$ to be the \emph{Gram matrix} of the vectors $f(1), \ldots, f(n)$, so that $(M_f)_{i,j} = \inprod{f(i)}{f(j)}$.
    
    The concept of orthonormal representations goes back to a seminal paper of Lov\'{a}sz \cite{L79}, who used them to define a graph parameter now known as the Lov\'{a}sz $\vartheta$-function. The $\vartheta$-function of a graph $G$ has several equivalent definitions. Here we list the ones that we shall use later.
    \begin{definition} \label{defn:theta}
        Let $G$ be a graph with vertex set $[n]$. The \emph{$\vartheta$-function of $G$}, denoted $\vartheta(G)$, can be defined in the following ways, which are shown to be equivalent in \cite{L79}.
        \begin{enumerate}
            \item \label{defn:one}
                $\vartheta(G)$ is the maximum, over all orthonormal representations $f$ of the complement graph $\overline{G}$, of the largest eigenvalue of the Gram matrix $M_f$.
            \item \label{defn:two}
                $\vartheta(G)$ is the maximum of $1 - \lambda_1(A) / \lambda_n(A)$, over all $n \times n$ real symmetric matrices $A$ such that $A_{i,j} = 0$ if $ij \in E(G)$ or $i = j$.\footnote{In \cite{L79}, Lov\'{a}sz forgets to include the assumption that $A$ is symmetric and $A_{i,i} = 0$ for all $i$ to his statement of Theorem 6, but it is clear that this is what he intended.}
            \item \label{defn:three}
                 $\vartheta(G)$ is the minimum, over all orthonormal representations $f$ of $G$ and all unit vectors $x$, of \\$\max_{v \in V(G)}{\inprod{x}{f(v)}^{-2}}$.
            \item \label{defn:four}
                 $\vartheta(G)$ is the maximum, over all orthonormal representations $f$ of the complement graph $\overline{G}$ and all unit vectors $x$, of $\sum_{v \in V(G)}{\inprod{x}{f(v)}^2}$.
        \end{enumerate}
    \end{definition}
    \noindent Lov\'{a}sz originally introduced the notion of the $\vartheta$-function in order to bound the Shannon capacity of a graph, and since then, the combinatorial and algorithmic applications of the Lov\'{a}sz $\vartheta$-function have been studied extensively, see e.g.\ Knuth \cite{K94}.
    
    Given a graph $G$, let us define the \emph{minimum semidefinite rank} of $G$, denoted $\msr(G)$, to be the minimum $d$ such that there exists an orthonormal representation of $G$ in $\R^d$. Note that $\msr(G)$ can be seen as a vector generalization of the chromatic number of $\overline{G}$, see \cite{HPSWM10}. Indeed, by assigning a standard basis vector of $\R^{\chi\left(\overline{G}\right)}$ to each vertex of a given color, one can see that $\msr(G) \leq \chi\left(\overline{G}\right)$. In the same paper where he introduced the $\vartheta$-function, Lov\'{a}sz \cite[Theorem 11]{L79} showed that 
    \[\vartheta(G) \leq \msr(G).\]
     
    Various notions of the minimum rank of a graph have been studied in the literature, see Fallat and Hogben \cite{FH14} for a survey. Note that an equivalent way to define the minimum semidefinite rank of a graph $G$ is as the minimum rank of a positive semidefinite matrix $M$ such that $M_{i,i} = 1$ for all $i$ and $M_{i,j} = 0$ if $ij \notin E(G)$. Dropping the positive semidefinite assumption, we arrive at the notion of minrank, which has applications in theoretical computer science, see Golovnev, Regev, and Weinstein \cite{GRW18} for references. In particular, it is related to important problems on the complexity of arithmetic circuits \cite{CPR00}. 

\subsection{A geometric problem of Lov\'{a}sz}
    
    One very interesting application of the Lov\'{a}sz $\vartheta$-function is to the following geometric problem posed by Lov\'asz and first studied by Konyagin \cite{K81}.
    \begin{quote}
        What is the maximum $\Delta_n$, of the length $\left| \left| \sum_{i=1}^{n}{x_i} \right| \right|$, over all $d$ and all unit vectors $x_1, \ldots, x_n \in \R^d$ such that among any three, there is at least one pair of orthogonal vectors? 
    \end{quote}
    Konyagin \cite{K81} gave upper and lower bounds on $\Delta_n$, in particular showing that $\Delta_n \leq \frac{3}{2} n^{2/3}$. Then Kashin and Konyagin \cite{KK83} improved the lower bound to within a logarithmic factor of the upper bound, and finally, Alon \cite{A94} was able to give an asymptotically tight construction showing that $\Delta_n = \Theta(n^{2/3})$. Note that if we define $L(G)$ to be the maximum of $\left| \left| \sum_{v \in V(G)}{f(v)} \right| \right|$ over all orthonormal representations $f$ for $G$, then the above problem is equivalent to asking for the maximum of $L(G)$ over all triangle-free graphs $G$ on $n$ vertices. The following claim, whose proof we defer to \Cref{subsec:theta}, connects $L(G)$ to $\vartheta(G)$ and $\vartheta\left(\overline{G}\right)$.
    \begin{claim} \label{L-theta}
        For any graph $G$ on $n$ vertices, we have
        \[ \frac{n}{\sqrt{\vartheta(G)}} \leq L(G) \leq \sqrt{n \vartheta\left(\overline{G}\right)}.\]
        Moreover, if $G$ is vertex-transitive, then $L(G) =  \sqrt{n \vartheta\left(\overline{G}\right)}$.
    \end{claim}

    For graphs $G,H$ we say that $G$ is \emph{$H$-free} if $G$ does not contain a copy of $H$ as a subgraph. Generalizing from a triangle to an arbitrary $H$, let us now define $\lambda(n, H)$ to be the maximum value of $\vartheta\left(\overline{G}\right)$ over all $H$-free graphs $G$ on $n$ vertices. Although in this paper we study only $\lambda(n,H)$, we remark that roughly speaking, \Cref{L-theta} would allow one to translate these results to the corresponding geometric problem of finding the maximum of $L(G)$ over all $H$-free graphs $G$ on $n$ vertices, especially because the constructions we consider are either Cayley graphs, which are vertex-transitive, or are very similar to Cayley graphs. Indeed, for $H = K_3$, Konyagin's argument for the upper bound on $\Delta_n$ can be adapted to obtain $\lambda(n, K_{3}) \leq O\left(n^{1/3}\right)$, and since Alon's construction for the lower bound on $\Delta_n$ is vertex-transitive, \Cref{L-theta} implies that $\lambda(n, K_3) \geq \Omega \left( n^{1/3} \right)$, so that we have $\lambda(n, K_3) = \Theta(n^{1/3})$. Generalizing to larger cliques, it is known that 
    \[ \Omega\left( n^{1 - O(1/\log{t}) } \right) \leq  \lambda(n, K_t) \leq O\!\left(n^{1 - 2/t}\right),  \] 
    where Alon and Kahale \cite{AK98} proved the upper bound and Feige \cite{F95} proved the lower bound.

    Another way to generalize forbidding a triangle is to forbid longer cycles. Indeed, Alon and Kahale \cite{AK98} also showed that for any $t$, if $G$ is a graph on $n$ vertices having no odd cycle of length at most $2t+1$, then $\theta\left(\overline{G}\right) \leq 1 + (n-1)^{1/(2t+1)}$. Our first contribution is a generalization of this upper bound to graphs that have no cycle of length exactly $2t+1$.
    \begin{theorem} \label{thm:odd-cycle-upper}
        For all $n, t \geq 1$ we have $\lambda(n, C_{2t+1}) \leq O\!\left( t \, n^{1/(2t+1)} \right)$.
    \end{theorem}

    We say that a graph $G$ has an \emph{optimal spectral gap} if $|\lambda_i(A)| \leq O\!\left(\sqrt{\lambda_1(A)}\right)$ for $2 \leq i \leq n$, where $A$ is the adjacency matrix of $G$ and $\lambda_1(A) \geq \ldots \geq \lambda_n(A)$ are its eigenvalues. Alon and Kahale also noted that their bound for graphs having no odd cycles of length at most $2t+1$, is tight via a modification (see, e.g., \cite{KS06} section 3, example 10) of the construction of Alon \cite{A94}. The key properties that make such a construction useful are that it is regular, dense, and has an optimal spectral gap. Indeed, a dense graph with an optimal spectral gap has an adjacency matrix with a large ratio of $| \lambda_1(A) / \lambda_n(A) |$, which by Definition \ref{defn:two} of the $\vartheta$-function leads to a good lower bound for $\theta\!\left(\overline{G}\right)$. For a graph $H$, the Tur\'{a}n number $\ex(n, H)$ is the maximum number of edges in an $H$-free graph $G$ on $n$ vertices. For bipartite $H$ such as $C_4, C_6, C_{10}, K_{2,t}$ and $K_{t, (t-1)! + 1}$, there are known constructions of $H$-free graphs $G$ that attain good lower bounds for the Tur\'{a}n number, i.e.\ $|E(G)| \geq \Omega\!\left( \ex(n, H) \right)$. Interestingly, most of these construction are also known to have optimal spectral gaps (see \cite{KS06} section 3, examples 6,7,12). Since they are regular with degree on the order of $\ex(n,H)/n$, it follows from the previous discussion that in such cases
    \[
        \lambda(n, H) \geq \vartheta\!\left(\overline{G}\right) \geq \Omega\left( \sqrt{\frac{\ex(n,H)}{n}} \right).
    \]
    In \Cref{subsec:theta}, we prove the following theorem by showing that the graphs discussed above have optimal spectral gaps.
    \begin{theorem} \label{thm:lower-bounds}
        Let $n \geq 1$.
        \begin{enumerate}
            \item For all $t \in \{4,6,10\} \cup (2\N + 1)$, we have $\lambda(n, C_t) \geq \Omega\left(n^{1/t}\right)$.
    
            \item For all $t \geq 2$, we have $\lambda\left(n, K_{2,t}\right) \geq \Omega\left( t^{1/4} n^{1/4}\right)$.
            
            \item For all $t \geq 3$, we have $\lambda\left(n, K_{t,(t-1)! + 1}\right) \geq \Omega\left( n^{\frac{1}{2}(1 - 1/t)} \right)$.
        \end{enumerate}
    \end{theorem}
    
    Since the Tur\'{a}n number can sometimes provide a lower bound for $\lambda(n,H)$, one might wonder if it can also provide an upper bound. If $H$ is a graph such that we can remove a vertex to obtain a tree and we have $\ex(n, H) \leq O(n^{1+\alpha})$ for some $\alpha > 0$, then we are able to obtain such an upper bound.
    \begin{theorem} \label{thm:Turan-upper-bound}
        Let $h \geq 1$ and let $H$ be a connected graph on $h$ vertices, containing a vertex $v$ with $H \backslash v$ being a tree. Furthermore, suppose that there exist $c, \alpha$ with $0 < \alpha \leq 1$ and $c \geq 1$ such that $\ex(n,H) \leq c n^{1 + \alpha}$ for all $n \geq 1$. Then for all $n \geq 1$, it holds that
        \[ \lambda(n, H) \leq 20 \cdot \frac{\sqrt{c h}}{\alpha} \cdot n^{\alpha/2}. \]
    \end{theorem}
    \noindent Now define $\theta_{t, s}$ to be the graph consisting of $s$ internally disjoint paths of length $t$ between a pair of vertices, and note that in particular $\theta_{t, 2} = C_{2t}$ and $\theta_{2, s} = K_{2,s}$. Since $\theta_{t, s}$ consists of a tree together with an additional vertex, we will use \Cref{thm:Turan-upper-bound} together with known upper bounds on Tur\'{a}n numbers to obtain the following corollary.
    \begin{corollary} \label{cor:Turan-upper-bound}
        Let $n \geq 1$. For all $t, s \geq 2$, we have $\lambda\!\left(n, \theta_{t, s} \right) \leq O\left( t^2 s^{1 - 1/(2t)} n^{1/(2t)} \right)$.
        In particular, for all $s,t \geq 2$, we have 
        \begin{align*}
            \lambda(n, C_{2t}) \leq O\!\left( t^2 n^{1/(2t)} \right) \qquad \qquad
            \lambda\!\left(n, K_{2, s} \right) \leq O\!\left( s^{3/4} n^{1/4} \right).
        \end{align*}
    \end{corollary}
    \begin{rem}
		The upper bound for $\lambda(n,C_{2t})$ can be improved to $O\left( t \, n^{1/(2t)} \right)$ using the proof technique from \Cref{thm:odd-cycle-upper}, see \Cref{thm:even-cycle-upper} in the appendix for details.
    \end{rem}
    
    \noindent Note that the lower bounds for $C_t$ and $K_{2,t}$ given in \Cref{thm:lower-bounds} have corresponding upper bounds via \Cref{thm:odd-cycle-upper} and \Cref{cor:Turan-upper-bound}, which are tight up to the constants depending on $t$. Unfortunately, since $K_{t,(t-1)!+1}$ for $t \geq 3$ is not a tree together with a vertex, we are able to obtain only a weak upper bound in this case.
    \begin{theorem} \label{thm:bipartite-upper-bound}
        For all $s \geq t \geq 2$,  there exists a constant $c_{t,s}$ such that
        \[
            \lambda\!\left(n, K_{t,s}\right) \leq c_{t,s} \, n^{1 - 2/t + 1/\left(t \, 2^{t-1}\right)}. 
        \]
    \end{theorem}
    
\subsection{Almost orthogonal vectors}

    Upon hearing about the results of Kashin and Konyagin \cite{KK83} towards Lov\'{a}sz's problem, Erd\H{o}s asked the following related question (see Ne\v{s}et\v{r}il and Rosenfeld \cite{N99} for a historical summary):
    
    \begin{quote}
        What is the maximum, $\alpha(d)$, of the number of vectors in $\R^d$ such that among any three distinct vectors there is at least one pair of orthogonal vectors?
    \end{quote}
    Rosenfeld \cite{R91} called such vectors \emph{almost orthogonal}. By taking two copies of each of the vectors from a basis in $\R^d$, we obtain $2d$ almost orthogonal vectors. Erd\H{o}s believed that a construction with more than $2d$ vectors does not exist, and indeed Rosenfeld showed that $\alpha(d) = 2d$ (see Deaett \cite{D11} for a short and nice proof that is slightly more general). 
    
    After his initial question was resolved, Erd\H{o}s further asked what happens if we replace 3 by a larger integer $t$. F\"{u}redi and Stanley \cite{FS92} defined $\alpha(d,t)$ to be the maximum number of vectors in $\R^d$ such that, among any $t+1$ distinct vectors, some pair is orthogonal. By considering $t$ orthogonal bases in $\R^d$, we obtain $\alpha(d,t) \geq dt$, and Erd\H{o}s asked whether equality holds. F\"{u}redi and Stanley proved that it does not always hold by showing that $\alpha(4,5) \geq 24$, and conjectured that there exists a constant $c$ such that $\alpha(d, t) < (dt)^{c}$. This conjecture was later also proven to be false by Alon and Szegedy \cite{AS99}, who showed that for some constant $\delta > 0$ and $t$ large enough, $\alpha(d, t) \geq d^{ \frac{\delta \log{t}}{\log{\log{t}}} }$.

    One can see that Erd\H{o}s's question is almost equivalent to asking for the minimum of $\msr(G)$ over all $K_{t+1}$-free graphs $G$ on $n$ vertices. The difference is that Erd\H{o}s was asking for the vectors to be distinct, while an orthonormal representation of a graph may label multiple vertices with the same vector. Nonetheless, we define $\rho(n,H)$ to be the minimum of $\msr(G)$ over all $H$-free graphs $G$ on $n$ vertices. Some further motivation for studying $\rho(n,H)$ comes from Pudl\'{a}k \cite{P02}, who, inspired by questions in circuit complexity, studied the minrank and minimum semidefinite rank of graphs without a cycle of given length. More recently, Haviv \cite{H18, H-F18} studied the minrank and Lov\'{a}sz $\vartheta$-function, in particular using the probabilistic method, in order to construct graphs with large minrank and whose complements are $H$-free.

    We note that the aforementioned results now take the form $\rho(n,K_{3}) = \lceil n/2 \rceil$, and 
    \[ \rho(n, K_{t+1}) \leq n^{\frac{\log{\log{t}}}{\delta \log{t}} } \]
    for some constant $\delta > 0$, $t$ sufficiently large, and an infinite number of values of $n$. Surprisingly, for $t$ fixed and $n$ large, it seems that the best known lower bound on $\rho(n, K_{t+1})$ is just what one gets from Ramsey theory: if $n > \binom{d+t}{t} \geq R(d+1,t+1)$ then any $K_{t+1}$-free graph on $n$ vertices has an independent set of size $d+1$, and therefore cannot have an orthonormal representation in $\R^d$. Since $\binom{d+t}{t} = O\!\left(d^t\right)$, we may conclude that $\rho(n, K_{t+1}) \geq \Omega(n^{1/t})$. Making use of Alon and Kahale's result \cite{AK98} that $\lambda(n,K_k) \leq O(n^{1-2/k}) $, we give a small improvement to this lower bound.
    \begin{theorem} \label{clique-rank}
        There exists a constant $\delta > 0$ such that for all $t \geq 3$ and $n \geq 1$, 
        $\rho(n, K_t) \geq \delta n^{3/t}$.
    \end{theorem}

    In the previous section, we saw that another way to generalize a question for triangle-free graphs is to forbid a longer cycle. Pudl\'{a}k \cite{P02} (Theorem 10) gave a case-based proof showing that there exists $c > 0$ such that $\rho(n, C_5) \geq c n$. Taking $t-1$ copies of each vector of an orthonormal basis in $\R^d$ gives an orthonormal representation of the graph consisting of $d$ cliques of size $t-1$, which implies 
    \[ \rho(n,C_{t}) \leq \lceil n/(t-1) \rceil. \] 
    Inspired by Erd\H{o}s, we may ask if equality holds. Unlike before, we show that the answer turns out to be yes, in particular improving and generalizing Pudl\'{a}k's result.
    \begin{theorem} \label{cycle-rank}
        For all $t \geq 3, n \in \N$ we have $\rho(n,C_{t}) = \lceil n/(t-1) \rceil$.
    \end{theorem} 
    \noindent Indeed, this follows from the following more general result, which holds for all connected graphs $H$ containing a vertex whose removal leaves a tree.
    \begin{theorem} \label{general-rank}
        Let $t \geq 1$ and let $H$ be a connected graph such that $V(H) = T \cup \{v\}$ where $H[T]$ is a tree on $t$ vertices. Then for all $n \geq 1$, $\rho(n,H) = \lceil n/t\rceil$.
    \end{theorem}

    \begin{rem}
        Our definition of $\msr(G)$ differs from the minimum semidefinite rank defined by Deaette \cite{D11}. Indeed, the representations $f : V(G) \rightarrow \C^d$ that he considers map into complex $d$-dimensional space, are allowed to map vertices to the 0 vector, and most importantly, must satisfy that $\inprod{f(u)}{f(v)} \neq 0$ if and only if $uv \in E(G)$. The last condition defines a \emph{faithful} representation, as studied by Lov\'{a}sz, Saks, and Schrijver \cite{LSS89}. Nevertheless,  \Cref{clique-rank,cycle-rank,general-rank} may be adapted to work with these alternate assumptions.
    \end{rem}
    
    We prove our results in the next two sections. We first prove \Cref{clique-rank,cycle-rank,general-rank} in \Cref{subsec:minrank}, and then proceed to prove the remaining results in \Cref{subsec:theta}. The final section of the paper contains some concluding remarks.

\section{Minimum semidefinite rank for $H$-free graphs} \label{subsec:minrank}

    To study the minimum semidefinite rank of a graph, we will need the following useful inequality. It goes back to \cite[p. 138]{B97} and its proof is based on a trick employed by Schnirelman in his work on Goldbach's conjecture \cite{S39}. For various combinatorial applications of this inequality, see, for instance, the survey by Alon \cite{A08}.

    \begin{lemma} \label{Schnirelmann}
        Let $M$ be a symmetric real matrix. Then $\tr(M)^2 \leq \rk(M)  \tr(M^2)$.
    \end{lemma}
    
    \begin{proof}
        Let $r$ denote the rank of $M$. Since $M$ is a symmetric real matrix, $M$ has precisely $r$ non-zero real eigenvalues $\lambda_1,\dots,\lambda_{r}$. Note that $\tr(M) = \sum_{i=1}^r\lambda_i$ and $\tr(M^2) = \sum_{i=1}^r \lambda_i^2$. Application of Cauchy--Schwarz yields the desired $ (\sum_{i=1}^r \lambda_i)^2 \leq r\sum_{i=1}^r \lambda_i^2$.
    \end{proof}

    Now we are ready to prove \Cref{clique-rank} and \Cref{general-rank}. \Cref{cycle-rank} follows immediately from \Cref{general-rank}.

    \begin{proof}[Proof of \Cref{clique-rank}]
        Let $\delta$ be a sufficiently small constant. We proceed by induction on $t$. For $t = 3$ we know that $\rho(n, K_3) = \lceil n/2 \rceil \geq \delta n$. 

        Now let $t \geq 3$ and let $G$ be a $K_{t+1}$-free graph on $n$ vertices. Let $f : V(G) \rightarrow \R^d$ be an orthonormal representation of $G$ in $\R^d$ with $M = M_f$ being the corresponding Gram matrix. We will make use of \Cref{Schnirelmann}. To this end, we shall upper bound $\tr(M^2)$. We have
        \begin{align*} \label{eqn:tr-M-squared}
        \begin{split}
            \tr\!\left(M^2\right) 
            &= \sum_{u \in V(G)}{ \left( \inprod{f(u)}{f(u)}^2 + \sum_{w \in N(u)}{\inprod{f(u)}{f(w)}^2} +  \sum_{w \notin N(u) \cup \{u\}}{\inprod{f(u)}{f(w)}^2} \right) } \\
            &= \sum_{u \in V(G)}{\left(1 + \sum_{w \in N(u)}{\inprod{f(u)}{f(w)}^2} \right)}.
        \end{split}
        \end{align*}
        Now fix $u \in V(G)$ and note that $G[N(u)]$ is $K_t$-free. Thus by the induction hypothesis, we have $d \geq \rho(|N(u)|, K_t) \geq \delta |N(u)|^{3/t}$. Since Alon and Kahale \cite{AK98} showed that there exists a constant $c$ such that $\lambda(n,K_t) \leq c \, n^{1-2/t}$, we have via Definition \ref{defn:four} of the $\vartheta$-function that
        \begin{equation*} \label{eqn:inprod-x-nbd}
            \sum_{w \in N(u)}{\inprod{f(u)}{f(w)}^2} 
            \leq \vartheta\!\left(\overline{G[N(u)]} \right) 
            \leq \lambda\!\left(|N(u)|, K_t \right) 
            \leq \lambda\!\left( (d/\delta)^{\frac{t}{3}}, K_t \right)
            \leq c \left( (d/\delta)^{\frac{t}{3}}\right)^{1 - \frac{2}{t}} 
            = c \cdot (d/\delta)^{\frac{t-2}{3}}.
        \end{equation*}
        Therefore, we conclude that $\tr\!\left( M^2 \right) \leq n \left( 1 + c \cdot (d/\delta)^{(t-2)/3} \right)$. Clearly $\tr(M) = n$ and $\rk(M) \leq d$, so applying \Cref{Schnirelmann} and dividing by $n$ yields
        \[ 
            n \leq d \left( 1 + c \cdot (d / \delta)^{(t-2)/3} \right) = d + c \cdot \delta^{-(t-2)/3} d^{(t+1)/3} \leq (d/\delta)^{(t+1)/3} 
        \]
        for $\delta$ a bit smaller than $1/c$. Thus $d \geq \delta n^{3/(t+1)}$ and since $G$ and $f$ were arbitrary, we conclude
        \[ \rho(n, K_{t+1}) \geq \delta n^{3/(t+1)}. \qedhere\] 
    \end{proof}

    \begin{proof}[Proof of \Cref{general-rank}]
        Let $d = \lceil n/t \rceil$ and let $G$ be a graph consisting of $d$ cliques of size $t$. Since $H$ is connected and has $t+1$ vertices, $G$ is clearly $H$-free. By assigning the standard basis vector $e_i \in \R^d$ to each vertex in the $i$-th clique for $i \in [d]$, we obtain an orthonormal representation of $G$ in $\R^d$, so that we conclude $\rho(n, H) \leq d = \lceil n/t \rceil$.
    
        For the lower bound, let $d= \rho(n,H)$ and let $G$ be an $H$-free graph on $n$ vertices that has an orthonormal representation $f$ in $\R^d$ with corresponding Gram matrix $M = M_f$. Next, we will use \Cref{Schnirelmann}. Note that, as in the proof of \Cref{clique-rank}, 
        \begin{align*} 
            \tr\!\left(M^2\right) 
            &= \sum_{u \in V(G)}{\left(1 + \sum_{w \in N(u)}{\inprod{f(u)}{f(w)}^2} \right)}.
        \end{align*}
        
        Now fix $u \in V(G)$ and observe that since $G$ has no copy of $H$, $G[N(u)]$ has no copy of some tree on $t$ vertices. It is well-known that in this case, $\chi(G[N(u)]) \leq t-1$, see e.g.\ corollaries 1.5.4 and 5.2.3 of Diestel \cite{D12}. Thus we can partition $N(u)$ into $t-1$ independent sets $B_1, \ldots, B_{t-1}$. Since $\{f(w) : w \in B_i\}$ is an orthonormal set of vectors, we have by Parseval's inequality that $\sum_{w \in B_i}{\inprod{f(w)}{v}^2} \leq ||v||^2$ for any $v \in \R^d$. In particular for $v = f(u)$, we therefore have
        \[ 
            \sum_{w \in N(u)}{\inprod{f(u)}{f(w)}^2} 
            = \sum_{i=1}^{t-1}{\sum_{w \in B_i}{ \inprod{f(u)}{f(w)}^2 } } 
            \leq \sum_{i=1}^{t-1}{||f(v)||^2} = t-1, 
        \]
        and thus
        \[ 
            \tr \!\left( M^2 \right) \leq \sum_{v \in V(G)}\left(1 + t - 1 \right) = nt.
        \]
        Clearly we have $\tr(M) = n$ and $\rk(M) \leq d$, so that by \Cref{Schnirelmann} we obtain $n^2 \leq d n t$. Thus we conclude that $d \geq n/t$ and so $\rho(n,H) = d \geq \lceil n/t \rceil$, as desired.
    \end{proof}

\section{Lov{\'a}sz $\vartheta$-function for $H$-free graphs}  \label{subsec:theta}  
    
    \begin{proof}[Proof of \Cref{L-theta}]
        Let $f$ be an orthonormal representation of $G$ that attains the maximum in the definition of $L(G)$, and denote its Gram matrix by $M_f$. Let $\allonesvector$ denote the all 1's column vector (here and later all of our vectors will be column vectors). We have that 
        \[ 
            L(G)^2 = \left| \left| \sum_{v \in V(G)}{f(v)}\right| \right|^2 
            = \sum_{u,v \in V(G)} \inprod{f(u)}{f(v)} 
            = \allonesvector^\intercal M_f \allonesvector
            \leq n \vartheta\left(\overline{G}\right),
        \]
        where the last inequality follows from Definition \ref{defn:one} of the $\vartheta$-function.
        
        For the other direction, let $f^*$ be an orthonormal representation of $G$ and $x$ be a unit vector that together attain the minimum in Definition \ref{defn:three} of $\theta(G)$. We therefore have that $\vartheta(G) \geq \inprod{x}{f^*(v)}^{-2}$ for all $v \in V(G)$. By changing the sign of $f^*(v)$ if necessary, we can ensure that $\inprod{x}{f^*(v)} \geq \vartheta(G)^{-1/2}$ for all $v \in V(G)$, so that by Cauchy--Schwarz we obtain 
        \[ 
            L(G) \geq ||x|| \left| \left| \sum_{v \in V(G)}{f^*(v)} \right| \right| \geq \inprod{x}{\sum_{v \in V(G)}{f^*(v)}} \geq \frac{n}{ \sqrt{\vartheta(G)}}. 
        \]
        Moreover, Lov\'{a}sz \cite[Theorem 8]{L79} showed that every vertex-transitive graph $G$ satisfies
        $\vartheta(\overline{G})\vartheta(G) = n$, in which case the upper and lower bounds for $L(G)$ coincide. Thus if $G$ is vertex-transitive, we conclude 
        \[ L(G) = \sqrt{n \vartheta(\overline{G})}. \qedhere\]
    \end{proof}

    In order to prove \Cref{thm:odd-cycle-upper} about $C_{2t+1}$-free graphs, we will need the following result proved implicitly by Erd\H{o}s, Faudree, Rousseau, and Schelp \cite{EFRS78}. It allows us to bound the chromatic number of the set of vertices at a fixed distance from a given vertex, for any graph without a cycle of prescribed length.
    \begin{lemma} \label{lem:chi-bound}
        Let $G$ be a graph having no cycle of length $k$ and let $i \leq \lfloor (k-1)/2 \rfloor$. Fix a vertex $u_0$ in $G$ and define $A_i = \{ u \in V(G) : d(u,u_0) = i \}$ to be the set of vertices at a distance of exactly $i$ from $u_0$. Then the induced subgraph $G[A_i]$ satisfies $\chi(G[A_i]) \leq k-2$.
    \end{lemma}
    
    \begin{proof}
        In the proof of Theorem 1 of \cite{EFRS78}, Erd\H{o}s, Faudree, Rousseau, and Schelp show that if $G$ does not contain a cycle of length $k$ and $i \leq \lfloor (k-1)/2 \rfloor$, then one can assign $k-2$ labels to the vertices of $A_i$ so that no two vertices having the same label are adjacent. Hence $\chi(G[A_i]) \leq k-2$.
    \end{proof}

    \begin{proof}[Proof of \Cref{thm:odd-cycle-upper}]
        Let $f$ be an orthonormal representation of $G$ maximizing the largest eigenvalue of the corresponding Gram matrix $M = M_f$. Let $\lambda_1 \geq \ldots \geq \lambda_n$ be the eigenvalues of $M$ and observe that by Definition \ref{defn:one} of the $\vartheta$-function, $\vartheta\left(\overline{G}\right) = \lambda_1$. Now note that $\tr(M^{2t+1}) = \sum_{i=1}^{n}{\lambda_i^{2t+1}}$ and that $\lambda_i \geq 0$ for all $i$ since $M$ is positive semidefinite. Thus we have $\lambda_1^{2t+1} \leq \sum_{i=1}^{n}{\lambda_i^{2t+1}} = \tr\left(M^{2t+1}\right)$, and hence $\vartheta\left(\overline{G}\right) \leq \tr\left(M^{2t+1}\right)^{1/(2t+1)}$. Therefore it will be enough for us to show that $\tr\left(M^{2t+1}\right) \leq (6t)^{2t} n$. 
        
        For convenience, given vertices $u_0, u_1, \ldots, u_k$, we define 
        \begin{equation*} \label{eqn:walk-notation}
            W(u_0, \ldots, u_k) 
            = \prod_{i=1}^{k}{\inprod{f(u_{i-1})}{f(u_i)}} 
            = f(u_0)^{\intercal} \left(\prod_{i=1}^{k-1} f(u_i) f(u_i)^{\intercal}\right) f(u_k)
        \end{equation*}   
        and note that $W(u_0, u_1, \ldots, u_{2t}, u_0) = 0$ whenever $u_0 u_1 \ldots u_{2t} u_0$ is not a closed walk in $G$, i.e.\ whenever one of the pairs $u_0 u_1, \ldots, u_{2t}u_0$ is a non-edge in $G$. Moreover, if $u_0 u_1 \ldots u_{2t} u_0$ form a closed walk in $G$, then $d(u_0, u_i) \leq t$ for all $i$, so if we define $N^t(u_0) = \{v \in G : d(v,u_0) \leq t\}$
        to be the set of vertices at a distance of at most $t$ from $u_0$, we obtain 
        \begin{align*}
            \tr\left(M^{2t+1}\right) 
            = \sum_{u_0, u_1, \ldots, u_{2t} \in V(G)} W(u_0, u_1, \ldots, u_{2t}, u_0) 
            = \sum_{u_0 \in V(G)} \sum_{u_1, \ldots, u_{2t} \in N^t(u_0)} W(u_0, u_1, \ldots, u_{2t}, u_0).
        \end{align*}
        Thus if we define 
        \[ Y(u_0) := \sum_{u_1, \ldots, u_{2t} \in N^t(u_0)} W(u_0, u_1, \ldots, u_{2t}, u_0)\]
        for $u_0 \in V(G)$, then it suffices for us to show that $Y(u_0) \leq (6t)^{2t}$ for all $u_0$, since we may then conclude 
        \[ \tr\left( M^{2t+1} \right) = \sum_{u_0 \in V(G)}{Y(u_0)} \leq (6t)^{2t} n.\]
        
        To bound $Y(u_0)$, we use \Cref{lem:chi-bound}. For any $r \leq t$, define $A_r = \{ u \in V(G) : d(u,u_0) = r \}$ to be the set of vertices at a distance of exactly $r$ from $u_0$. Since $G$ has no cycle of length $2t+1$, we have by \Cref{lem:chi-bound} that $\chi(G[A_r]) \leq 2t$, and so we let $\{B(r,1), \ldots, B(r, 2t)\}$ be a partition of $A_r$ into $2t$ independent sets. Note that for every closed walk $u_0 \ldots u_{2t} u_0$, if we let $d_i = d(u_0, u_i)$ denote the distance from $u_0$ to $u_i$, then $|d_{i+1} - d_i| \le 1$. Thus we obtain
        \[
            Y(u_0) 
            = \sum_{\substack{d_1, \ldots, d_{2t} : \\ d_1 = 1, \,\, |d_{i+1} - d_i| \le 1}} \,\, \sum_{a_1, \ldots, a_{2t} \in [2t]} \,\, \sum_{\substack{u_1, \ldots, u_{2t} : \\ u_i \in B(d_i, a_i)}} W(u_0, u_1, \ldots, u_{2t}, u_0) .
        \]
        Now since each $B(r,s)$ is an independent set, it follows that $\{ f(u) : u \in B(r,s) \}$ is an orthonormal set of vectors. Moreover, observe that $P_{r,s} := \sum_{u \in B(r,s)}f(u) f(u)^{\intercal}$ is precisely the orthogonal projection onto the subspace spanned by $\{ f(u) : u \in B(r,s) \}$. Thus for any $d_1,\ldots, d_{2t}$ such that $d_1 = 1, |d_{i+1}-d_i| \le 1$ for all $i$ and for any $a_1, \ldots, a_{2t} \in [2t]$, we have
        \begin{align*}
            \sum_{\substack{u_1, \ldots, u_{2t}: \\ u_i \in B(d_i, a_i)}} W(u_0, u_1, \ldots, u_{2t}, u_0) 
            &= \sum_{\substack{u_1, \ldots, u_{2t}: \\ u_i \in B(d_i, a_i)}} f(u_0)^{\intercal} \left(\prod_{i=1}^{2t} f(u_i) f(u_i)^{\intercal}\right) f(u_0) \\
            &= f(u_0)^{\intercal} \left(\prod_{i = 1}^{2t} \sum_{u_i \in B_{d_i, a_i}} f(u_i) f(u_i)^{\intercal}\right) f(u_0) \\
            &= f(u_0)^{\intercal} \left(\prod_{i=1}^{2t} P_{d_i, a_i} \right) f(u_0),
        \end{align*}
        and since any orthogonal projection $P$ satisfies $||Pv|| \leq ||v||$, we may apply Cauchy--Schwarz to obtain
        \begin{align*}
            f(u_0)^{\intercal} \left(\prod_{i=1}^{2t} P_{d_i, a_i} \right) f(u_0) 
            = \inprod{f(u_0)}{ \left(\prod_{i=1}^{2t} P_{d_i, a_i} \right) f(u_0) } 
            &\leq || f(u_0) || \left| \left| \left(\prod_{i=1}^{2t} P_{d_i, a_i} \right) f(u_0) \right| \right| \\
            &\leq || f(u_0) || || f(u_0) || \\
            &= 1.
        \end{align*}
        Since there are at most $3^{2t}$ sequences of integers $(d_1, \ldots, d_{2t})$ such that $d_1 = 1$ and $|d_{i+1} - d_i| \leq 1$ for all $i$, we therefore conclude
        \begin{align*}
            Y(u_0)
            &= \sum_{\substack{d_1, \ldots, d_{2t} \\ d_1 = 1, \,\, |d_{i+1} - d_i| \le 1}} \,\, \sum_{a_1, \ldots, a_{2t} \in [2t]} \sum_{\substack{u_1, \ldots, u_{2t}: \\ u_i \in B(d_i, a_i)}} W(u_0, u_1, \ldots, u_{2t}, u_0) \\
            &\leq \sum_{\substack{d_1, \ldots, d_{2t} \\ d_1 = 1, \,\, |d_{i+1} - d_i| \le 1 }} \,\, \sum_{a_1, \ldots, a_{2t} \in [2t]} 1 \\
            &\leq 3^{2t} (2t)^{2t} = (6t)^{2t}. \qedhere
        \end{align*}
    \end{proof}
    
    We now cite known constructions of $C_t$-free, $K_{2,t}$-free, or $K_{t,(t-1)!+1}$-free graphs with many edges and optimal spectral gaps, in order to obtain \Cref{thm:lower-bounds}. Note that some of the graphs described below have loops on some of their vertices, so to get a simple graph these loops should be removed. Since this only affects the adjacency matrix by subtracting a diagonal matrix with $1$s and $0$s on the diagonal, one can deduce from Weyl's interlacing inequality that the eigenvalues only change by at most 1, not affecting the asymptotic bounds obtained below.

    \begin{proof}[Proof of \Cref{thm:lower-bounds}]
    For the $C_{2t+1}$, $C_4$, $C_6$, $C_{10}$, and $K_{t,(t-1)!+1}$-free graph constructions and their spectral properties discussed below, see section 3 of the survey on pseudo-random graphs \cite{KS06} by Krivelevich and Sudakov.
    
    As previously mentioned, Alon and Kahale \cite{AK98} note that a modification of Alon's construction \cite{A94} gives a graph with an optimal spectral gap which is, in particular, $C_{2t+1}$-free for any fixed $t \geq 1$. For more details, see \cite{KS06} section 3, example 10. Indeed, for any $k$ such that $2^k - 1$ is not divisible by $4t + 3$, the construction yields a $2^{k-1}(2^{k-1}-1)$-regular graph $G$ on $n = 2^{(2t+1)k}$ vertices which is $C_{2t+1}$-free such that all eigenvalues of its adjacency matrix except the largest are bounded in absolute value by $O(2^k)$. The adjacency matrix $A$ of such a graph therefore has largest eigenvalue $\lambda_1(A) = 2^{k-1}(2^{k-1}-1)$ and all other eigenvalues bounded in absolute value by $O(2^k)$. Applying Definition \ref{defn:two} of $\vartheta\left(\overline{G}\right)$ to the adjacency matrix of $G$, and using the fact that the smallest eigenvalue of $G$ is negative (as the trace of the adjacency matrix is $0$), we thus conclude
    \[ \lambda(n, C_{2t+1}) \geq \vartheta\left(\overline{G}\right) \geq 1 + \frac{2^{k-1}(2^{k-1}-1)}{O\!\left(2^k\right)} = \Omega\left(n^{1/(2t+1)}\right).\]
    
    The construction of a $C_4$-free graph $G$ with an optimal spectral gap and many edges comes from the projective space over a finite field of order $q = p^{\alpha}$ where $p$ is a prime and $\alpha$ is an integer, see \cite{KS06} section 3, example 6. It has $n = q^2 + q + 1$ vertices, is $(q+1)$-regular (so $\lambda_1 = q+1$), and all of its eigenvalues beside the largest are in absolute value equal to $\sqrt{q}$. Therefore, we obtain as above that 
    \[ \lambda(n, C_4) \geq \vartheta\left(\overline{G}\right) \geq 1 + \frac{q+1}{\sqrt{q}} = \Omega\!\left( n^{1/4} \right).\]
    The $C_6$-free graph and the $C_{10}$-free graph with optimal spectral gaps and many edges are the polarity graphs of a generalized $4$-gon and $6$-gon respectively, see \cite{KS06} section 3, example 7. As above, these graphs yield the bounds $\lambda(n, C_6) \geq \Omega\!\left( n^{1/6} \right)$ and  $\lambda(n, C_{10}) \geq \Omega\!\left( n^{1/10} \right)$.
     
    The $K_{t, (t-1)!+1}$-free graph $G$ with an optimal spectral gap and many edges is called a projective norm graph, see \cite{KS06} section 3, example 12. For a prime $p$, $G$ has $p^t - p^{t-1}$ vertices, is $(p^{t-1}-1)$-regular, and all eigenvalues besides the largest are in absolute value at most $p^{(t-1)/2}$. Thus we obtain
    \[ \lambda(n, K_{t,(t-1)!+1}) \geq \vartheta\left(\overline{G}\right) \geq 1 + \frac{p^{t-1}-1}{p^{(t-1)/2}} = \Omega\!\left( n^{\frac{1}{2}\left(1 - 1/t\right)} \right). \]
    
    The following construction of a $K_{2,t+1}$-free graph with many edges is due to F\"{u}redi \cite{F96}. As he did not show that this construction has an optimal spectral gap, we prove it below. Let $q$ be a prime power such that $t$ divides $q-1$ and let $\Fe$ be a finite field of order $q$. Let $h \in \Fe$ be an element of order $t$ and let $H = \{1, h, h^2, \ldots, h^{t-1}\}$. Define the equivalence relation on $\Fe \times \Fe \setminus \{(0,0)\}$ by $(a,b) \sim (a',b')$ iff there exists $c \in H$ such that $(a', b') = c \cdot (a, b)$. Let $\langle a,b \rangle$ denote the equivalence class of $(a,b)$ under the relation $\sim$. Now define $G$ to be the graph whose vertices are the equivalence classes $\left(\Fe \times \Fe \setminus \{(0,0)\} \right) / \sim$ such that there is an edge between $\langle a,b \rangle $ and $\langle a',b' \rangle $ iff $a a' + b b' \in H$.
    
    Each equivalence class has $t$ elements, and therefore $G$ has $n = (q^2 - 1)/t$ vertices. Moreover, for each vertex $(a,b) \in \Fe \times \Fe \setminus \{(0,0)\}$, there are $q$ solutions $(x,y)$ to the equation $ax + by = c$ for any $c \in H$, and therefore $\langle a,b \rangle$ has degree $tq/t = q$. 
    Now let $\langle a, b \rangle, \langle a', b' \rangle$ be a pair of distinct vertices and consider their common neighborhood. To determine its size, we must determine the number of solutions $(x,y)$ to the equations 
    \begin{align*}
        ax + by &= d\\
        a'x + b'y &= e 
    \end{align*}
    where $d,e \in H$. If there exists $c$ such that $a' = ca, b' = cb$, then the equations have no solutions, since otherwise we would have $e = c a x + c b y = c d$, which would imply that $c \in H$, contradicting the fact that $\langle a,b \rangle \neq \langle a',b' \rangle$. Thus $\langle a,b \rangle$ and $\langle a',b' \rangle$ have no common neighbors in this case. Otherwise if there does not exist $c$ such that $a' = ca, b' = cb$, then the matrix 
    $\begin{pmatrix}
    a & b\\
    a' & b'
    \end{pmatrix}$
    is invertible and hence the system of equations has a unique solution $(x,y)$ for each choice of $d,e \in H$. As there are $t^2$ choices for $d$ and $e$, we obtain a total of $t^2$ solutions, which implies that there are $t^2 / t = t$ vertices in the common neighborhood of $\langle a,b \rangle$ and $\langle a', b' \rangle$. Thus $G$ has no copy of $K_{2,t+1}$.
    
    Now let $A$ be the adjacency matrix of $G$, indexed by the vertices $\langle a, b \rangle$, and consider $A^2$. Since $G$ is $q$-regular, the diagonal entries of $A^2$ will all be $q$.  The off-diagonal entry $A^2_{\langle a, b\rangle, \langle a', b' \rangle}$ is the number of common neighbors of $\langle a, b\rangle$ and $\langle a', b' \rangle$, which by the previous discussion is either $0$ or $t$ depending on whether or not there exists $c$ such that $a' = ca, b' = cb$. Thus if we let $Q$ be the $\{0,1\}$ matrix indexed by the vertices of $G$ so that $Q_{\langle a, b\rangle, \langle a', b' \rangle} = 1$ iff $\langle a, b\rangle$ and $\langle a', b' \rangle$ have no common neighbors, then we have 
    \[ A^2 = (q - t)I + t J - tQ\]
    where $I$ is the identity matrix and $J$ is the all-ones matrix. Now for any given $\langle a, b \rangle$, observe that we must have $c \in (\Fe \setminus \{0\}) \setminus H$ in order for $a' = ca, b' = cb$ to yield $(a',b') \neq (0,0)$ such that $\langle a',b' \rangle \neq \langle a,b \rangle$. This gives $q - 1 - t$ choices for $c$ and therefore there are exactly $(q-1-t)/t$ many vertices $\langle a', b' \rangle$ that have no common neighbors with $\langle a,b \rangle$, so that $Q$ is a matrix with $(q-1-t)/t$ ones in each row. By the Perron--Frobenius theorem, the largest eigenvalue of $Q$ is $\lambda_1(Q) = (q-1-t)/t$ with eigenvector $\allonesvector$, and all other eigenvalues satisfy $|\lambda_i(Q)| \leq \lambda_1(Q)$ and have eigenvectors which are orthogonal to $\allonesvector$. $J$ has largest eigenvalue $\lambda_1(J) = n$ also with the eigenvector $\allonesvector$ and any vector orthogonal to $\allonesvector$ is an eigenvector of $J$ with eigenvalue $0$. Therefore, any eigenvector of $Q$ is also an eigenvector of $A^2$ which implies that for all $i \geq 2$,
    \[ |\lambda_i(A^2)| \leq q - t + t \cdot \frac{q-1-t}{t} = 2q - 2t - 1.\]
    Now since $G$ is $q$-regular, the largest eigenvalue of $A$ is $q$, and all other eigenvalues are square roots of eigenvalues of $A^2$. Thus we conclude 
    \[ \max_{i\geq2} |\lambda_i(A)| \leq \sqrt{2q - 2t - 1}. \]
    Finally, applying Definition \ref{defn:two} of $\vartheta\left( \overline{G} \right)$ with the matrix $A$, we obtain
    \[ \lambda(n, K_{2,t+1}) \geq \vartheta\!\left(\overline{G}\right) \geq 1 - \frac{\lambda_1(A)}{\lambda_n(A)} \geq 1 + \frac{q}{\sqrt{2q - 2t - 1}} = \Omega\!\left( t^{1/4} n^{1/4} \right). \qedhere\]
    \end{proof}

    We now give a proof of \Cref{thm:Turan-upper-bound}, using an approach similar to that which was used by Alon and Kahale to prove $\lambda(n, K_{t}) \leq O(n^{1 - 2/t})$ in \cite{AK98}.
    
    \begin{proof}[Proof of \Cref{thm:Turan-upper-bound}]
        We proceed by induction on $n$. For $n=1$ the claim holds trivially. Now suppose $n \geq 2$ and let $G$ be an $H$-free graph on $n$ vertices. Define $U = \{ v \in V(G) : d(v) \leq 4 c \cdot n^{\alpha} \}$ and $W = V(G) \backslash U$. It follows from Definition \ref{defn:four} of the $\vartheta$-function that
        $\vartheta\!\left(\overline{G}\right) \leq \vartheta\!\left(\overline{G[U]}\right) + \vartheta\!\left(\overline{G[W]}\right)$.
        Moreover, observe that
        \[ 4 c \cdot n^{\alpha} |W| \leq \sum_{v \in W}{d(v)} \leq 2 \ex(n, H) \leq 2 c \cdot n^{1+\alpha}, \]
        so $|W| \leq n/2$, and hence by the induction hypothesis \[
        \vartheta\!\left(\overline{G[W]}\right) \leq \lambda(n/2,H) \leq 20 \cdot \frac{\sqrt{c h}}{\alpha} \cdot \!\left( \frac{n}{2} \right)^{\alpha/2}.
        \]
        
        It remains to bound $\vartheta\!\left(\overline{G[U]}\right)$. To this end let $f$ be an orthonormal representation of $G[U]$ maximizing the largest eigenvalue $\lambda_1(M)$ of the corresponding Gram matrix $M = M_f$. By Definition \ref{defn:one} of the $\vartheta$-function, we have $\vartheta\!\left(\overline{G[U]}\right) = \lambda_1(M)$. Now fix $u \in U$ and define $N'(u) = \{ w \in U : uw \in E(G) \}$ to be the neighborhood of $u$ in $G[U]$. Since $G[U]$ has no copy of $H$, we have that $N'(u)$ induces no copy of the tree $T$. Therefore, by the same argument as in the proof of \Cref{general-rank}, $N'(u)$ can be partitioned into at most $h$ independent sets, each corresponding to a set of orthonormal vectors. Thus by Parseval's inequality, $\sum_{w \in N'(u)}{\inprod{f(u)}{f(w)}^2} \leq h$. Since $|N'(u)| \leq d(u) \leq 4c \cdot n^{\alpha}$, we conclude via Cauchy--Schwarz that $\sum_{w \in N'(u)}{|\inprod{f(u)}{f(w)}|} \leq \sqrt{4c \cdot n^{\alpha} h}$. Note that $\lambda_1(M) \leq \max_{u \in U}{\sum_{w \in U}{|\inprod{f(u)}{f(w)}|}}$, and thus
        \[
            \vartheta\!\left(\overline{G[U]}\right) 
            \leq \max_{u \in U}{\sum_{w \in U}{|\inprod{f(u)}{f(w)}|}} 
            = \max_{u \in U}{\!\left(1 + \sum_{w \in N'(u)}{|\inprod{f(u)}{f(w)}|}\right)} 
            \leq 1 + \sqrt{4c \cdot n^{\alpha} h} 
            \leq 3 \sqrt{c \cdot n^{\alpha} h}.
        \]
        Putting everything together, we have
        \[ 
            \vartheta\!\left(\overline{G}\right) \leq \vartheta\!\left(\overline{G[U]}\right) + \vartheta\!\left(\overline{G[W]}\right) 
            \leq 3 \sqrt{c \cdot n^{\alpha} h} + 20 \cdot \frac{\sqrt{ch}}{\alpha} \cdot \left( \frac{n}{2} \right)^{\alpha/2} 
            = \sqrt{c \cdot n^{\alpha} h} \cdot \!\left(3 + \frac{20}{\alpha} \cdot \!\left( \frac{1}{2}\right)^{\alpha/2} \right).
        \]
        Now to complete the proof, we use the fact that $e^{-x} \leq 1 - x/2$ for $0\leq x \leq 1$ to conclude
        \[ 
            3 + \frac{20}{\alpha} \cdot \!\left( \frac{1}{2}\right)^{\alpha/2} 
            \leq 3 + \frac{20}{\alpha} \cdot \!\left(1 - \frac{\ln(2) \alpha}{4} \right)
            \leq \frac{20}{\alpha}. \qedhere
        \] 
    \end{proof}
    
    \Cref{cor:Turan-upper-bound} will now follow from known upper bounds on Tur\'{a}n numbers.
    
    \begin{proof}[Proof of \Cref{cor:Turan-upper-bound}]
        Recently, Bukh and Tait \cite{BT18} showed that $\ex\!\left(n, \theta_{t,s}\right) \leq O\!\left(t s^{1 - 1/t} n^{1 + 1/t}\right)$, generalizing the well-known upper bounds $\ex\left(n, C_{2t}\right) \leq O\left(t n^{1 + 1/t}\right)$ due to Bondy and Simonovits \cite{BS74}, and $\ex\left(n, K_{2,t}\right) \leq O\left(t n^{1 + 1/t}\right)$ due to F\"{u}redi \cite{F96}. Since $\theta_{t,s}$ consists of a tree together with an additional vertex, we may apply \Cref{thm:Turan-upper-bound} to obtain the desired upper bounds on $\lambda(n, H)$.
    \end{proof}
    \begin{rem}
        Bukh and Jiang \cite{BJ17} recently improved the upper bound on $\ex(n, C_{2t})$ to $O\!\left(\sqrt{t} \log{t} \, n^{1 + 1/t}\right)$ for $n$ sufficiently large relative to $t$. Using \Cref{thm:Turan-upper-bound}, this implies $\lambda(n, C_{2t}) \leq O\!\left( t^{7/4} \sqrt{\log{t}} \, n^{1/(2t)}\right)$. Nonetheless, in the appendix we show how to obtain the better bound $\lambda(n, C_{2t}) \leq O\!\left(t \, n^{1/(2t)}\right)$ via a different argument.
    \end{rem}

    \Cref{thm:bipartite-upper-bound} will follow from an argument similar to that of \Cref{thm:Turan-upper-bound}, except that we will have to replace the result that the chromatic number of a neighborhood is bounded, with a bound on the $\vartheta$-function of a neighborhood which will be obtained inductively.
    
    \begin{proof}[Proof of \Cref{thm:bipartite-upper-bound}]
        We proceed by induction on $n$ and $t$, where $c_{t,s}$ will be defined recursively. For $s \geq t=2$, let $c_{2,s}$ be the constant such that $\lambda(n, K_{2,s}) \leq c_{2,s} s^{3/4} n^{1/4}$ as given by \Cref{cor:Turan-upper-bound}. Now suppose $s \geq t \geq 3$. For $n = 1$, the claim trivially holds for $c_{t,s} \geq 1$. Now let $n \geq 2$.
        
        K\"{o}vari, S\'{o}s, and Tur\'{a}n \cite{KST54} showed that there exists a constant $a_{t,s}$ such that $\ex(n, K_{t,s}) \leq a_{t,s} n^{2 - 1/t}$. As in the proof of \Cref{thm:Turan-upper-bound}, define $U = \{ v \in V(G) : d(v) \leq 4 a_{t,s} n^{1 - 1/t} \}$, $W = V(G) \backslash U$, and observe that by Definition \ref{defn:four} of the $\vartheta$-function,
        $\vartheta\!\left(\overline{G}\right) \leq \vartheta\!\left(\overline{G[U]}\right) + \vartheta\!\left(\overline{G[W]}\right)$.
        Moreover, observe that
        \[ 
            4 a_{t,s} n^{1 - 1/t} |W| \leq \sum_{v \in W}{d(v)} \leq 2 \ex(n, K_{t,s}) \leq 2 a_{t,s} n^{2 - 1/t}, 
        \]
        so $|W| \leq n/2$, and hence by the induction hypothesis 
        \[
            \vartheta\!\left(\overline{G[W]}\right) \leq \lambda(n/2,K_{s,t}) \leq c_{t,s} \left(\frac{n}{2}\right)^{1 - 2/t + 1/\left(t \, 2^{t-1}\right)}.
        \]
         To bound $\vartheta\!\left(\overline{G[U]}\right)$, let $f$ be an orthonormal representation of $G[U]$ maximizing the largest eigenvalue $\lambda_1(M)$ of the corresponding Gram matrix $M = M_f$, so that we have $\vartheta\!\left(\overline{G[U]}\right) = \lambda_1(M)$. Now fix $u \in U$ and let $N'(u) = \{ w \in U : uw \in E(G) \}$ be the neighborhood of $u$ in $G[U]$. Note that $G[U]$ has no copy of $K_{t-1,s}$, so that via Definition \ref{defn:four} of the $\vartheta$-function and induction, we have 
         \[
             \sum_{w \in N'(u)}{\inprod{f(u)}{f(w)}^2} \leq \vartheta\!\left(\overline{G[U]}\right) \leq \lambda(|N'(u)|, K_{t-1,s}) \leq c_{t-1,s} \, |N'(u)|^{1 - 2/(t-1) + 1/\left((t-1) \, 2^{t-2}\right)}.
         \]
         Thus using the fact that $|N'(u)| \leq 4 a_{t,s} n^{1-1/t}$ and applying Cauchy--Schwarz, we conclude
         \begin{align*}
            \sum_{w \in N'(u)}{|\inprod{f(u)}{f(w)}|} 
            &\leq \sqrt{|N'(u)| \cdot c_{t-1,s} \cdot |N'(u)|^{1 - 2/(t-1) + 1/\left((t-1) \, 2^{t-2}\right)}}\\
            &\leq 4\sqrt{c_{t-1,s}} \cdot  a_{t,s}^{1 - 1/(t-1) + 1/\left((t-1) \, 2^{t-1}\right)} \cdot n^{1 - 2/t + 1/\left(t \, 2^{t-1}\right)} .
        \end{align*}
         As in the proof of \Cref{thm:Turan-upper-bound}, we therefore obtain
        \begin{align*}
            \vartheta\!\left(\overline{G[U]}\right) 
            \leq  \max_{u \in U}{\sum_{w \in U}{|\inprod{f(u)}{f(w)}|}} 
            &= \max_{u \in U}{\!\left(1 + \sum_{w \in N'(u)}{|\inprod{f(u)}{f(w)}|}\right)}\\
            &\leq \!\left(1 + 4\sqrt{c_{t-1,s}} \cdot a_{t,s}^{1 - 1/(t-1) + 1/\left((t-1) \, 2^{t-1}\right)} \right) n^{1 - 2/t + 1/\left(t \, 2^{t-1}\right)}.
        \end{align*}
        Thus if we set 
        \[ 
            c_{t,s} = \frac{1 + 4\sqrt{c_{t-1,s}} \cdot a_{t,s}^{1 - 1/(t-1) + 1/\left((t-1) \, 2^{t-1}\right)}}{1 - (1/2)^{1 - 2/t + 1/\left(t \, 2^{t-1}\right)}},
        \]
        then we conclude the desired result
        \[ 
            \vartheta\!\left(\overline{G}\right) \leq \vartheta\!\left(\overline{G[U]}\right) + \vartheta\!\left(\overline{G[W]}\right) 
            \leq c_{t,s} \, n^{1 - 2/t + 1/\left(t \, 2^{t-1}\right)}. \qedhere
        \]
    \end{proof}
    
\section{Concluding remarks}

We have seen that for $H \in \{C_{2t+1}, C_4, C_6, C_{10}, K_{2,t}\}$ fixed and $n$ large, \Cref{thm:lower-bounds} and \Cref{cor:Turan-upper-bound} provide bounds on $\lambda(n,H)$ that are asymptotically tight. However, the lower bound in \Cref{thm:lower-bounds} for $\lambda(n, K_{t,s})$ with $s \geq t\geq 3$ does not match the upper bound obtained in \Cref{thm:bipartite-upper-bound}, so determining the correct asymptotic dependence on $n$ is an interesting problem. Indeed, for $n \gg t \rightarrow \infty$, we have
\[
    1/2 - o(1)\leq \log_{n}{\lambda(n, K_{t,s})} \leq 1 - o(1),
\]
where the lower bound is coming from graphs with optimal spectral gaps that are almost extremal for the Tur\'{a}n number, so that we cannot hope to do better with such constructions. On the other hand, we know 
\[
    1 - o(1) \leq \log_{n}{\lambda(n, K_{t})} \leq 1 - o(1),
\]
for $n \gg t \rightarrow \infty$, and it would therefore be interesting to determine if the asymptotic behavior of $\lambda(n, H)$ is different for $H = K_t$ versus $H = K_{t,s}$.

For $H = K_{2,t}$, even though we know the asymptotic behavior of $\lambda(n, H)$, we are only able to show that
\[ 
    \Omega\!\left(t^{1/4} n^{1/4}\right) \leq \lambda(n, K_{2,t}) \leq O\!\left(t^{3/4} n^{1/4}\right),
\] 
so it would be interesting to determine the correct dependence of $\lambda(n, K_{2,t})$ on $t$.

\vspace{0.25cm}
\noindent
{\bf Acknowledgment.}\,
We would like to thank Boris Bukh and Chris Cox for stimulating discussions.

\appendix
\section*{Appendix}

Here we give an improved bound for $\lambda(n, C_{2t})$ using the same approach as in \Cref{thm:odd-cycle-upper}. The argument is more complicated because \Cref{lem:chi-bound} does not work for vertices at a distance of $t$ from a given vertex.

\begin{theorem} \label{thm:even-cycle-upper}
    For all $t \ge 2$, we have $\lambda(n, C_{2t}) \leq  12 t \, n^{1/(2t)}$.
\end{theorem}
\begin{proof}
    Let $f$ be an orthonormal representation of $G$ maximizing the largest eigenvalue of the corresponding Gram matrix $M = M_f$. As in the proof of \Cref{thm:odd-cycle-upper}, it will suffice to show that $\tr(M^{2t}) \le (12t)^{2t} n$. Recall the notations $W(u_0, \ldots, u_k)$ and $N^t(u_0)$ introduced in the proof of \Cref{thm:odd-cycle-upper}. We have
    \begin{align*}
        \tr\left(M^{2t}\right) 
        &= \sum_{u_0, u_1, \ldots, u_{2t-1} \in V(G)} W\!\left(u_0, \ldots, u_{2t-1}, u_0\right),
    \end{align*}
    where $W\!\left(u_0, \ldots, u_{2t-1}, u_0\right) = 0$ unless $u_0 u_1 \ldots u_{2t-1} u_0$ forms a closed walk in $G$. Moreover, since $G$ is $C_{2t}$-free, any such closed walk must satisfy $u_i = u_j$ for some $0 \le i < j \le 2t-1$. Observe that this can happen either if $u_1, \ldots, u_{2t-1} \in N^{t-1}(u_0)$, or if $d(u_0, u_i) = d(u_0, u_{2t-i}) = i$ for every $i \in [t]$, and $u_i = u_{2t-i}$ for some $i \in [t-1]$. Thus if we define
    \begin{align*} 
        Y(u_0) & = \sum_{u_1, \ldots, u_{2t-1} \in N^{t-1}(u_0)} W\!\left(u_0, \ldots, u_{2t-1}, u_0\right), \\
        Z(u_0) & = 
            \sum_{\substack{ 
                u_1, \ldots, u_{2t-1} : \\
                d(u_0, u_i) \,=\, d(u_0, u_{2t-i}) \,=\, i \,\, \forall i \in [t], \\
                \text{$u_i = u_{2t-i}$ for some $i \in [t-1]$}.
            }} W\!\left(u_0, \ldots, u_{2t-1}, u_0\right),
    \end{align*}
    then we have
    \begin{align*}
        \tr\left(M^{2t}\right) 
        &= \sum_{u_0 \in V(G)} \left(Y(u_0) + Z(u_0)\right).
    \end{align*}
    We will show that $Y(u_0) \le (6t)^{2t}$ and $Z(u_0) \le (4t)^{2t}$ for every $u_0 \in V(G)$, which will complete the proof of the theorem.
    
    To prove that $Y(u_0) \le (6t)^{2t}$ for every vertex $u_0$, one can repeat the argument from the proof of \Cref{thm:odd-cycle-upper}. We now turn to the task of upper-bounding $Z(u_0)$. For non-empty $I \subseteq [t-1]$, we define
    \begin{align*} 
        Z_I(u_0) & = 
            \sum_{\substack{ 
                u_1, \ldots, u_{2t-1} : \\
                d(u_0, u_i) \,=\, d(u_0, u_{2t-i}) \,=\, i \,\, \forall i \in [t], \\
                u_i \,=\, u_{2t-i}  \,\, \forall i \in I.
            }} W\!\left(u_0, \ldots, u_{2t-1}, u_0\right),
    \end{align*}
    and observe that by the inclusion-exclusion principle,
    \[
        Z(u_0) = \sum_{I \subseteq [t-1],\, I \neq \emptyset}(-1)^{|I|} Z_I(u_0).
    \]
    It thus suffices to show that $|Z_I(u_0)| \le (2t)^{2t}$ for every non-empty $I \subseteq [t-1]$. For vertices $u_0, u_{\ell}, u_k$ with $d(u_0, u_{\ell}) = \ell,\,\, d(u_0, u_k) = k$, where $0 \le \ell < k \le t$, define
    \begin{align*}
        S(u_0, u_{\ell}, u_k) & = 
        \left(\sum_{\substack{ 
                u_{\ell+1}, \ldots, u_{k-1} : \\
                d(u_0, u_i) \,=\, i \,\,\, \forall \ell \,<\, i \,<\, k. 
            }}  W\!\left(u_\ell, \ldots, u_{k}\right) \right)^2.
    \end{align*}
    Let $I \subseteq [t-1]$ be non-empty, and write $I = \{\alpha_1, \ldots, \alpha_{m-1}\}$, where $\alpha_1 < \ldots < \alpha_{m-1}$. Also let $\alpha_0 = 0$ and $\alpha_m = t$. Now observe that
    \begin{align*}
        Z_I(u_0) 
        &= \sum_{\substack{u_{\alpha_1},\ldots, u_{\alpha_m}: \\ d(u_0, u_{\alpha_i}) = \alpha_i \,\,\, \forall i \in [m]}} \,\,\prod_{i = 1}^m S\!\left(u_0, u_{\alpha_{i-1}}, u_{\alpha_i}\right) \\
        &= \sum_{\substack{u_{\alpha_1}: \\ d(u_0, u_{\alpha_1}) = \alpha_1}} \left(S(u_0, u_0, u_{\alpha_1}) \sum_{\substack{u_{\alpha_2}: \\ d(u_0, u_{\alpha_2}) = \alpha_2}} \left( S(u_0,u_{\alpha_1},u_{\alpha_2}) \cdots \sum_{\substack{u_{\alpha_m}: \\ d(u_0, u_{\alpha_m}) = \alpha_m}} S(u_0, u_{\alpha_{m-1}}, u_{\alpha_m}) \right) \ldots \right).
    \end{align*}
    Note that we have $\alpha_i - \alpha_{i-1} \leq t-1$ for all $i \in [m]$. We shall show that
    \[
        \sum_{u_k: \,\, d(u_0, u_k) = k} S(u_0, u_{\ell}, u_k) \le (2t)^{2(k-\ell)}
    \] 
    for all $k, \ell$ and $u_0, u_{\ell}$ such that $d(u_0, u_{\ell}) = \ell < k \le t$ and $k-\ell \leq t-1$. Since it is clear by definition that $S(u_0, u_{\ell}, u_k) \ge 0$ for every $u_0, u_{\ell}, u_k$, we may then conclude that $0 \le Z_I(u_0) \le (2t)^{2t}$, as required.
    
    The remainder of the proof is very similar to the proof of \Cref{thm:odd-cycle-upper}. Given $\ell, k$ and $u_0, u_{\ell}$ as above, let $A_i = \{u \in V(G) : d(u_0, u) = i, \, d(u_{\ell}, u) = i - \ell\}$ for $\ell < i \le k$. Since $d(u_{\ell}, u) = i-\ell \leq t-1$ for all $u \in A_i$, we may apply \Cref{lem:chi-bound} to conclude that $\chi(G[A_i]) \leq 2t$, and so we let $\{ B(i, 1), \ldots, B(i, 2t) \}$ be a partition of $A_i$ into $2t$ independent sets. Also observe that if $d(u_0, u_{\ell}) = \ell$, $d(u_0, u_i) = i$, and $u_0 \ldots u_{\ell} \ldots u_i$ is a walk in $G$, then $d(u_{\ell}, u_i ) = i - \ell$ so that $u_i \in A_i$. Therefore we obtain
    \begin{align*}
        &\sum_{u_k : \,\, d(u_0, u_k) = k} S(u_0, u_{\ell}, u_k)\\
        =\,& \sum_{\substack{u_{\ell+1}, \ldots, u_{k-1}, u_k: \,\, u_i \in A_i\\ w_{\ell+1}, \ldots, w_{k-1}: \,\, w_i \in A_i }} W\!\left(u_{\ell}, \ldots, u_k, w_{k-1}, \ldots, w_{\ell+1}, u_{\ell}\right)\\
        =\,& \sum_{\substack{a_{\ell+1},\ldots, a_k \in [2t] \\ b_{\ell+1}, \ldots, b_{k-1} \in [2t]}}\,\,
        \sum_{\substack{u_{\ell+1}, \ldots, u_k, w_{\ell+1}, \ldots, w_{k-1}: \\ u_i \in B(i, a_i) \\ w_i \in B(i, b_i)}} W\!\left(u_{\ell}, \ldots, u_k, w_{k-1}, \ldots, w_{\ell+1}, u_{\ell}\right) \\
        =\,& \sum_{\substack{a_{\ell+1},\ldots, a_k \in [2t] \\ b_{\ell+1}, \ldots, b_{k-1} \in [2t]}}\,\,
        \sum_{\substack{u_{\ell+1}, \ldots, u_k, w_{\ell+1}, \ldots, w_{k-1}: \\ u_i \in B(i, a_i) \\ w_i \in B(i, b_i) }} \,\,f(u_{\ell})^{\intercal} \left(\prod_{i \,=\, \ell+1}^k f(u_i) f(u_i)^{\intercal} \prod_{i \,=\, k-1}^{\ell+1} f(w_i) f(w_i)^{\intercal}\right) f(u_{\ell}) \\
        =\,& \sum_{\substack{a_{\ell+1}, \ldots, a_k \in [2t] \\ b_{\ell+1}, \ldots, b_{k-1} \in [2t]}}\,\,
        f(u_{\ell})^{\intercal} \left(\prod_{i \,=\, \ell+1}^k \left(\sum_{u_i \in B(i, a_i)} f(u_i) f(u_i)^{\intercal}\right) \prod_{i \,=\, k-1}^{\ell+1} \left( \sum_{w_i \in B(i, b_i)}\,\, f(w_i) f(w_i)^{\intercal} \right) \right) f(u_{\ell}).
    \end{align*}
    Thus if we define $P_{i, a} = \sum_{u_i \in B(i, a)} f(u_i) f(u_i)^{\intercal}$, then 
    \begin{align*}
        \sum_{u_k : \,\, d(u_0, u_k) = k} S(u_0, u_{\ell}, u_k) 
        = \sum_{\substack{a_{\ell+1}, \ldots, a_k \in [2t] \\ b_{\ell+1}, \ldots, b_{k-1} \in [2t]}} \inprod{f(u_{\ell})}{\left(\prod_{i \,=\, \ell+1}^k P_{i, a_i}\right)\left(\prod_{i \,=\, k-1}^{\ell+1} P_{i, b_i}\right)f(u_{\ell})}.
    \end{align*}
    Note that $P_{i, a}$ is an orthogonal projection onto the space spanned by $\{f(u_j): j \in B(i, a)\}$, and thus $\norm{P_{i,a} v} \le \norm{v}$ for every vector $v$. It follows by Cauchy--Schwarz that
    \begin{align*}
        \left|\sum_{u_k : \,\, d(u_0, u_k) = k} S(u_0, u_{\ell}, u_k) \right| 
        &\le \sum_{\substack{a_{\ell+1}, \ldots, a_k \in [2t] \\ b_{\ell+1}, \ldots, b_{k-1} \in [2t]}} \left|\inprod{f(u_{\ell})}{\left(\prod_{i \,=\, \ell+1}^k P_{i, a_i}\right)\left(\prod_{i \,=\, k-1}^{\ell+1} P_{i, b_i}\right)f(u_{\ell})} \right| \\           
        & \le \sum_{\substack{a_{\ell+1}, \ldots, a_k \in [2t] \\ b_{\ell+1}, \ldots, b_{k-1} \in [2t]}} \norm{f(u_{\ell})} \cdot \norm{\left(\prod_{i \,=\, \ell+1}^k P_{i, a_i}\right) \left(\prod_{i \,=\, k-1}^{\ell+1} P_{i, a_i}\right) f(u_{\ell})} \\
        & \le \sum_{\substack{a_{\ell+1}, \ldots, a_k \in [2t] \\ b_{\ell+1}, \ldots, b_{k-1} \in [2t]}} \norm{f(u_{\ell})}^2 \le (2t)^{2(k-\ell)}.
    \end{align*}
    As explained above, this completes the proof that $0 \le Z_I(u_0) \le (2t)^{2t}$ for every vertex $u_0$ and every non-empty $I \subseteq [t-1]$, which completes the proof of the theorem.
\end{proof}

\end{document}